\documentclass{amsart}

\usepackage{amsmath,amsfonts,amssymb,amscd,verbatim}
\usepackage[arrow,matrix,cmtip,curve]{xy}
\usepackage{mathrsfs}  
\usepackage{xcolor}

\theoremstyle{definition}
\newtheorem{theorem}{Theorem}

\newtheorem{definition}[theorem]{Definition} 

\newtheorem{lemma}[theorem]{Lemma}
\newtheorem{proposition}[theorem]{Proposition}

\newtheorem{remark}[theorem]{Remark}

\newtheorem*{thm}{Theorem}

\DeclareMathOperator\gr{gr}
\DeclareMathOperator\id{id}

\newcommand\cnt{\mathcal Z}
\newcommand\inv{^{-1}}
\newcommand\iso{\cong}
\newcommand\kk{\Bbbk}

\newcommand\cS{\mathcal S}
\newcommand\cX{\mathcal X}
\newcommand\cY{\mathcal Y}

\newcommand\NN{\mathbb N}

\newcommand\I{\mathfrak I}

\newcommand\oA{\overline{A}}
\newcommand\oB{\overline{B}}
\newcommand\oQ{\overline{Q}}
\newcommand\ophi{\overline{\phi}}

\begin{document}

\title{Isomorphisms of graded path algebras}

\author{Jason Gaddis}
\address{Miami University, Department of Mathematics, 301 S. Patterson Ave., Oxford, Ohio 45056} 
\email{gaddisj@miamioh.edu}
\subjclass[2010]{Primary 16W20, 16W50}


\begin{abstract}
We prove that if two path algebras with homogeneous relations are isomorphic as algebras, then they are isomorphic as graded path algebras. This extends a result by Bell and Zhang in the connected case.
\end{abstract}

\maketitle

Given a family of (noncommutative) algebras, the isomorphism problem is to determine under what conditions there exists an isomorphism between two members of that family.
See \cite{BJ,BZ1,BZ2,CPWZ1,Giso1,Giso2,GH,LWZ1}.

In this note, we study the isomorphism problem for path algebras of quivers modulo homogeneous relations. Such algebras appear naturally in the study of (twisted) Calabi-Yau algebras \cite{Bock,BSW,GR1,RR2,RR1}.
In particular, we show that any isomorphism between two such algebras reduces to an isomorphism that respects the path-length filtration.

Throughout, $\kk$ is a field, all algebras are $\kk$-algebras, and $\dim=\dim_{\kk}$.
An algebra $A$ is $\NN$-graded if it has a $\kk$-vector space decomposition
$A = \bigoplus_{n=0}^\infty A_n$ such that $A_m \cdot A_n \subset A_{m+n}$.
Suppose $B$ is another $\NN$-graded algebra with decomposition $B=\bigoplus_{n=0}^\infty B_n$.
We say $A$ and $B$ are {\sf isomorphic as graded algebras} if there exists an algebra
isomorphism $\phi:A \rightarrow B$ such that $\phi(A_n)=B_n$.
The next theorem is a powerful tool in the study of isomorphism problems for graded algebras.

\begin{thm}[Bell, Zhang {\cite[Theorem 0.1]{BZ1}}]
Let $A$ and $B$ be two connected graded algebras finitely generated over $\kk$ in degree $1$.
If $A \iso B$ as ungraded algebras, then $A \iso B$ as graded algebras.
\end{thm}

A quiver $Q=(Q_0,Q_1)$ is a directed graph where $Q_0$ is the set of vertices
and $Q_1$ is the set of arrows. We assume throughout that $|Q_0|, |Q_1| <\infty$.
For $a \in Q_1$, we denote by $s(a)$ and $t(a)$ the source and target of $a$, respectively.
The $(u,v)$-entry of the {\sf adjacency matrix} $M_Q$
of $Q$ records the number of arrows from vertex $u$ to vertex $v$.
A path in $Q$ is a collection of arrows $p=a_0a_1\cdots a_n$ such that $s(a_k)=t(a_{k-1})$ for $k=1,\hdots,n$. In this case, $s(p)=s(a_0)$ and $t(p)=t(a_n)$. The path $p$ is said to be a cycle if $s(p)=t(p)$. Note that the number of paths in $Q$ from vertex $u$ to vertex $v$ of length $\ell$ is the $(u,v)$-entry of $(M_Q)^\ell$.

At each vertex $v \in Q_0$ there is a {\sf trivial path} $e_v$ with $s(e_v)=t(e_v)=v$.
A quiver is a semigroup under the operation of concatenation. That is, 
given paths $p=a_0a_1\cdots a_n$ and $q=b_0b_1\cdots b_m$, $p \cdot q=a_0\cdots a_nb_0\cdots b_m$ if $s(q)=t(p)$ and $p\cdot q=0$ otherwise. 
The path algebra $\kk Q$ is the set of finite linear combinations of paths
with multiplication extended linearly and identity element $1 = \sum_{v \in Q_0} e_v$.
This implies that the trivial paths form a complete set of orthogonal idempotents.

If $\I$ is an ideal of $\kk Q$ generated by elements of homogeneous path length at least one,
then $A=\kk Q/\I$ has the structure of an $\NN$-graded algebra when we assign $\deg(a)=1$ for all $a \in Q_1$.
We say the ideal $\I$ is {\sf homogeneous} if it is generated by homogeneous elements.
However, this grading ignores some of a path algebra's most significant structural properties.
Suppose $p \in \I$ is a homogeneous relation in path length.
Then we decompose $p$ as $p = \sum_{u,v \in Q_0} e_u p e_v$ where each summand $e_u p e_v \in \I$.
It will be convenient in this exposition to regard an element $r \in \kk Q$ as {\sf homogeneous} if all summands of $r$ have the same (path) length, source, and target.

\begin{remark}
If $\I$ is a homogeneous ideal of $\kk Q$ with a degree zero relation, then necessarily $\I$ contains $e_v$ for some trivial path $e_v$. Furthermore, $\I$ contains any path passing through vertex $v$. Then $\kk Q/\I \iso \kk Q'/\I'$ where $Q'$ is obtained from $Q$ by removing vertex $v$ and any path through it, and $\I'$ is the induced ideal on $\kk Q'$.
Similarly, if $\I$ contains a degree one relation from vertex $u$ to vertex $v$, then we may replace $Q$ by a quiver $Q'$ wherein we remove one of the nonzero summands of that relation.
Henceforth, we assume that homogeneous ideals $\I$ of a path algebra $\kk Q$ contain no degree zero or degree one relations.
\end{remark}

Let $A = \kk Q/\I$ and $B=\kk Q'/\I'$ with $\I,\I'$ homogeneous ideals.
Set $(A_k)_{uv}$ to be the paths of length $k$ (modulo relations in $\I$) from vertex $u$ to vertex $v$. Thus, we record $\dim A_k$ as a $(|Q_0| \times |Q_0|)$-matrix
where the values of $(\dim A_k)_{uv}$ are $\dim((A_k)_{uv})$ and similarly for $B$.
We say $A$ and $B$ are {\sf isomorphic as graded path algebras} if $M_{Q'}=P M_Q P\inv$ 
for some permutation matrix $P$ corresponding to $\sigma \in \cS_{|Q_0|}$ and 
there exists an algebra isomorphism $\phi:A \rightarrow B$ such that $\phi((A_k)_{uv})=(B_k)_{\sigma(u)\sigma(v)}$.
In Theorem \ref{thm.main}, we prove that if two graded path algebras
$A = \kk Q/\I$ and $B=\kk Q'/\I'$ are isomorphic as (ungraded algebras),
then they are isomorphic as graded path algebras.
Theorem \ref{thm.main} reduces to \cite[Theorem 0.1]{BZ1} when $|Q_0|=1$.

If $J$ is a not necessarily homogeneous ideal in $A$ with $|Q_0|=n$, then we write $\dim A/J = I_n$ when the induced map $A_0 \to A/J$ given by $A_0 = \kk e_1 + \cdots + \kk e_n$ is an isomorphism. We write $\dim J/J^2 = S$ where $S$ is an $n \times n$ matrix in which $S_{uv} = \dim ((e_u J e_v + J^2)/J^2)$.


\begin{definition}
Let $A=\kk Q/\I$ with $\I$ homogeneous, $|Q_0|=n$, and let $S$ be an $n \times n$ matrix with entries in $\NN$. An ideal $J$ of $A$ with the property $\dim A/J = I_n$ and $\dim J/J^2 = S$ is a {\sf codimension $I_n$ ideal of tangent dimension $S$}. Let $J_S(A)$ be the intersection of all such ideals in $A$.
If this intersection is empty, then we set $J_S(A)=A$.
\end{definition}

The {\sf irrelevant ideal} of $A=\kk Q/\I$, generated by the image of $Q_1$ in $A$, is denoted $A_+$. Note that $A_+$ is necessarily a codimension $I_n$ ideal of tangent dimension $M_Q$.

Given two matrices $S,S' \in M_n(\NN)$, we say $S>S'$ if $S_{uv} \geq S_{uv}'$ for all $u,v$ and there exists $u_0,v_0$ such that $S_{u_0v_0} > S_{u_0v_0}'$.


Next we generalize \cite[Lemma 5]{BZ1} to the case of path algebras.

\begin{lemma}
\label{lem.codim}
Let $A=\kk Q/\I$ with $\I$ homogeneous and set $M=M_Q$. If $S_{uv}> M_{uv}$ for some $u,v \in Q_0$, then $J_S(A)=A$ and if $\dim A_+/(A_+)^2 = N$, then $J_N(A) \subset A_+$.
\end{lemma}

\begin{proof}
Let $\{a_k^{uv}\}$ be a set of algebra generators of $A$, where $u,v$ denotes the source and target of each particular generator. Here $u$ and $v$ range over $Q_0$ while $k$ ranges from $1$ to $M_{uv}$.

Now if $J$ is an ideal of codimension $I_n$, 
then $J$ is generated by $\{x_k^{uv}\}$ where 
\[ x^{uv}_k = \begin{cases} 
a_k^{uv} & \text{ if } u \neq v \\ 
a_k^{uu}-\alpha_k^{uu} e_u & \text{ if } u=v, \end{cases}\]
with $\alpha_k^{uu} \in \kk$.
Hence $A$ is generated by $\{x_k^{uv}\}$ and $J^2$ is generated by $\{ x_k^{uv}x_\ell^{vw} \}$.
It follows that $(\dim (J/J^2))_{uv} \leq M_{uv}$ and so $\dim J/J^2 \leq M$.
\end{proof}

We will retain the notation $\{a_k^{uv}\}$ used in the previous lemma throughout.


For a homomorphism $\phi:A \rightarrow B$ of $\NN$-graded algebras,
we denote by $\phi_k(a)$ the degree $k$ component of the image $\phi(a)$.
Part of the next lemma generalizes \cite[Lemma 3.2]{KW}.

\begin{lemma}
\label{lem.verts}
Let $A = \kk Q/\I$ and $B=\kk Q'/\I'$ with $\I,\I'$ homogeneous ideals. If $\phi:A \to B$ is an isomorphism, then $|Q_0|=|Q_0'|$ and $M_{Q'} = PM_QP\inv$ for some permutation matrix $P$.
\end{lemma}
\begin{proof}
Let $M=M_Q$ and $N=M_{Q'}$.
Suppose $|Q_0|=n$ (resp. $|Q_0'|=m$) and denote the trivial orthogonal idempotents of 
$\kk Q$ (resp. $\kk Q'$) by $e_1,\hdots,e_n$ (resp. $f_1,\hdots,f_m$). Suppose
\[\phi(e_u) = \sum \alpha_{uv} f_v + \text{ (higher degree paths)}\]
where $\alpha_{uv} \in \kk$. Then 
\[ \sum \alpha_{uv} f_v = \phi_0(e_1) = \phi_0(e_1^2) = \sum_v \alpha_{uv}^2 f_v.\]
Thus, each $\alpha_{uv}=0$ or $1$. By the grading on $A$, and because $\I$ is homogeneous, it follows that at least one of the $\alpha_{uv}$ must be nonzero.

If $u \neq w$, then $0 = \phi_0(e_u e_w) = \sum \alpha_{uv}\alpha_{wv} f_v$.
Since there are no degree zero relations in $\I$ or $\I'$,
then $\phi_0(e_u)$ and $\phi_0(e_v)$ have no common nonzero summands amongst the $f_v$.
Hence, $n \leq m$ and by considering  $\phi\inv$ we have $m \leq n$, so $n=m$.
That is, $\phi_0$ determines a permutation $\sigma$ on the vertices of $Q$ so that $\phi_0(e_u) = f_{\sigma(u)}$. After renumbering vertices, we may assume that $\sigma=\id$.

Let $a \in Q_1$ with source $u$ and target $v$.
By the above, 
\begin{align}
\label{eq.triv}
\phi(e_u)=f_u + b_1 + \cdots + b_\ell
\end{align}
for some $b_i \in (B_+)^i$. Then
\[ \phi(a) = \phi(e_ua) = (f_u+ b_1 + \cdots + b_\ell) \phi(a).\]
Write $\phi(a) = \phi_0(a) + \cdots + \phi_d(a)$ with $\phi_d(a) \neq 0$.
If $b_\ell\phi_d(a) \neq 0$, then it has degree $d+\ell$, a contradiction. It follows by induction that $b_i\phi_d(a)=0$ for all $i$. Thus, $f_u \phi_d(a) = \phi_d(a)$.
Suppose $b_i\phi_k(a)=0$ for all $i$ and some $k$, $0 < k \leq d$. Then $f_u \phi_k(a)=\phi_k(a)$.
Moreover, $b_\ell\phi_{k-1}(a)$ has degree $\ell+k-1\geq k$ if it is nonzero, implying $b_\ell\phi_{k-1}(a)=0$. As before we can now conclude that 
$b_i\phi_{k-1}(a)=0$ for all $i$.
Induction again implies that $f_u\phi(a)=\phi(a)$. Similarly, $\phi(a)f_v = \phi(a)$.

Since $\phi$ is an isomorphism, then $J=\phi(A_+)$ is a codimension $I_n$ ideal of $B$.
By the previous paragraph, $M_{uv} = \dim((f_u J f_v + J^2)/J^2)$ and so by Lemma \ref{lem.codim}, $M_{uv} \leq N_{uv}$ for all $u,v$. Reversing the argument we get $N_{uv} \leq M_{uv}$, so $N=M$.
\end{proof}

Not every isomorphism $\phi:A \rightarrow B$ of graded path algebras
necessarily takes vertices to vertices. For example, there is an automorphism $\phi$ of the quiver
\[ \xymatrix{ ^1\bullet \ar[r]^a & \bullet^2}\]
defined by $\phi(e_1)=e_1+a$, $\phi(e_2)=e_2-a$, and $\phi(a)=a$.
However, it follows from Lemma \ref{lem.verts} that we may assume that, after reordering vertices, $\phi(e_u)=f_u$ for all $u \in Q_0$.

The next lemma proves the main theorem for quivers without loops.

\begin{lemma}
\label{lem.noloops}
Let $A = \kk Q/\I$ and $B=\kk Q'/\I'$ with $\I,\I'$ homogeneous ideals generated in degree at least two.
Suppose $Q$ and $Q'$ have no loops. If $A \iso B$ as (ungraded) algebras, then $A \iso B$ as graded path algebras.
\end{lemma}
\begin{proof}
By Lemma \ref{lem.verts}, $M_{Q'}=PM_QP\inv$ for some permutation matrix $P$.
After renumbering vertices we may assume $M=M_Q=M_{Q'}$. 
Let $\phi:A \rightarrow B$ be the given (ungraded) isomorphism.
Denote by $e_1,\hdots,e_n$ (resp. $f_1,\hdots,f_n$) the trivial idempotents of $\kk Q$ (resp. $\kk Q'$).
For each $1 \leq u,v \leq |Q_0|$, let $\{x_k^{uv}\}$ (resp. $\{y_k^{uv}\}$) 
denote the subset of generators of $A_1$ (resp. $B_1$) that lie in $e_uA_1e_v$ (resp. $f_uB_1f_v$).

The proof of Lemma \ref{lem.verts} asserts that if $a \in Q_1$ has source $u$ and target $v$, then for every $k$ we have that $\phi_k(a)$ is a linear combination of paths with source $u$ and target $v$. In particular, $f_u \phi_0(a) = \phi_0(a) = \phi_0(a)f_v$. If $u=v$, then $\phi_0(a)=\alpha f_u$ for some $\alpha \in \kk$. If $u\neq v$, then $\phi_0(a)=0$.
An identical argument may be applied to $\phi\inv$.

Set $J = \phi(A_+)$ and $K = \phi\inv(B_+)$.
Since $B_+$ is generated as an ideal by the $\{y_k^{uv}\}$, then by the grading on $B$,
\[ B_+/(B_+)^2 = \bigoplus_{u,v \in Q_0} \bigoplus_{k=1}^{M_{uv}} \kk \overline{y_k^{uv}}.\]
As $K$ is a codimension $I_n$ ideal of tangent dimension $M$, then by a linear change of variable,
\[ K_{uv} = \begin{cases}
	e_uA(x_1^{uv},\hdots, x_{M_{uv}}^{uv})Ae_v & \text{ if } u \neq v \\
	e_uA(x_1^{uu},\hdots, x_{M_{uu}}^{uu}-\alpha_u e_u)Ae_u & \text{ if } u = v
	\end{cases}
\]
for some $\alpha_u \in \kk$ and 
\[(K/K^2)_{uv} = \begin{cases}
	\bigoplus_{k=1}^{M_{uv}} \kk \overline{x_k^{uv}}  & \text{ if } u \neq v \\
	\bigoplus_{k=1}^{M_{uu}} \kk \overline{x_k^{uu}} \oplus \kk \overline{x_{M_{uu}}^{uu}-\alpha_u e_u} & \text{ if } u = v.
	\end{cases}\]
Thus, $\phi$ induces a $\kk$-linear isomorphism $\ophi: K/K^2 \rightarrow B_+/(B_+)^2$.

Set
\begin{align*}
	\cX &= \{ x_k^{uv}  : u \neq v \text{ or } (u=v \text{ and } k \neq M_{uu}) \} \\
	\cY &= \{ y_k^{uv}  : u \neq v \text{ or } (u=v \text{ and } k \neq M_{uu}) \}.
\end{align*}
Similarly, we let $\cX^c=\{ x_{M_{uu}}^{uu} \}$ and $\cY^c = \{ y_{M_{uu}}^{uu} \}$. After reordering vertices and a linear change of variable, there exists $(y_k^{uv})' \in (B_+)^2$ such that 
\begin{align}
\label{eq.phi}
\phi(x_k^{uv}) = \begin{cases}
	y_k^{uv} + (y_k^{uv})' & \text{ if } x_k^{uv} \in \cX \\
	\alpha_u f_u + y_k^{uv} + (y_k^{uv})' & \text{ otherwise}.
	\end{cases}
\end{align}

In case $\alpha_u = 0$ for some $u$, then there will be no loss in assuming $x_{M_{uu}}^{uu} \in \cX$ and $y_{M_{uu}}^{uu} \in \cY$.
Suppose $\alpha_u=0$ for all $u \in Q_0$. Then $J \subset B_+$ and because $J$ has codimension $I_n$ we have $J=B_+$. Thus, $\phi$ induces a graded algebra isomorphism from 
$\gr_{A_+} A := \bigoplus_{i=0}^\infty (A_+)^i/(A_+)^{i+1}$ to
$\gr_{B_+} B := \bigoplus_{i=0}^\infty (B_+)^i/(B_+)^{i+1}$,
which maps $e_u \mapsto f_u$.
Since $A \iso \gr_{A_+} A$ and $B \iso \gr_{B_+} B$ as graded algebras.
\end{proof}

The proof \cite[Theorem 1]{BZ1} relied heavily on the use of commutators.
Before proving the general theorem, we modify this technique for use in path algebras.

Keep the notation introduced in Lemma \ref{lem.noloops}.
Let $p$ be a path from vertex $u$ to vertex $v$ and note that we allow $u=v$ here.
Suppose $p$ contains some $a_w \in \cX^c$ at a vertex $w$ and write $p=p_1a_wp_2$ for paths $p_1,p_2$ where $s(p_1)=u$, $t(p_1)=s(p_2)=w$, and $t(p_2)=v$.
We call a term of the form $a_up_1-p_1a_w$ a {\sf $(u,w)$-commutator} and denote it $[p_1]$. Note that there is no ambiguity in this notation because for each vertex there is at most one loop in $\cX^c$ at that vertex. Moreover, $u$ and $w$ correspond to the source and target, respectively, of $p_1$.
Assume there is a loop $a_u \in \cX^c$ at vertex $u$. Then
\begin{align}
\label{eq.comm}
p 
	= (p_1a_w-a_up_1)p_2 + a_up_1p_2
	= a_up_1p_2 - [p_1].
\end{align}
We can also build iterative $(u,w)$-commutators of the form
$[ \cdots [ ~ [p_1]~] \cdots ]$.
Let $[\cX]$ (resp. $[\cY]$) denote the vector space of iterative commutators of elements in $\cX$ (resp. $\cY$) including trivial commutators, i.e., elements of $\cX$.

\begin{lemma}
\label{lem.comm1}
Let $p$ be a path in the $\{x_k^{uv}\}$ of degree $d$ with source $u$ and target $v$. Suppose there exists $a_u \in \cX^c$ with source $u$. Then $p=\sum_{s=0}^d a^s p_s$ with $p_s \in [\cX]$.
\end{lemma}
\begin{proof}
We may assume that $p$ contains a loop in $\cX^c$ since otherwise the result is trivial.
Write $p=p_1 a_{u_1}^{k_1} p_2 a_{u_2}^{k_2} \cdots p_\ell a_{u_\ell}^{k_\ell} p_{\ell + 1}$ where each $p_i \in \cX$, $a_{u_i} \in \cX^c$, and $k_i \in \NN$. Note that we allow $p_1$ and $p_{\ell + 1}$ to be trivial but assume that all other $p_i$ are nontrivial. 

Consider the case $\ell=1$ and write $p=p_1a_w^kp_2$ for some $p_1,p_2 \in [\cX]$. If $k=0$ then there is nothing to prove. Assume the conclusion holds for some $k \geq 0$. Applying \eqref{eq.comm} we have
\[
p	
	= p_1a_w^{k+1}p_2 \\
	= (a_up_1 - a_up_1 + p_1a_w)a_w^kp_2
	= (a_up_1 - [p_1])a_w^kp_2.
\]
Now we may apply the inductive hypothesis to $p_1a_w^kp_2$ and $[p_1] a_w^kp_2$.

Returning to the main result, assume the conclusion holds for some $\ell \geq 1$. Suppose $p = p_1 a_{u_1}^{k_1} p_2 a_{u_2}^{k_2} \cdots p_\ell a_{u_{\ell+1}}^{k_{\ell+1}} p_{\ell + 2}$. Let $d'$ be the degree of the path $p_1 a_{u_1}^{k_1} p_2 a_{u_2}^{k_2} \cdots p_\ell$. Then by the inductive hypothesis and after relabeling we have 
\[ p = \left( \sum_{s=0}^{d'} a_u^s \tilde{p}_s \right) a_{u_{\ell+1}}^{k_{\ell+1}} p_{\ell + 2}\]
for some $\tilde{p}_s \in [\cX]$. Applying the previous induction to each $\tilde{p}_s a_{u_{\ell+1}}^{k_{\ell+1}} p_{\ell + 2}$ gives the result.
\end{proof}

In Lemma \ref{lem.noloops}, we considered the case where $\alpha_u=0$ for all $u$. Next we consider the other extreme, where $\alpha_u \neq 0$ for all $u$.

\begin{lemma}
\label{lem.all-loops}
Let $A = \kk Q/\I$ and $B=\kk Q'/\I'$ with $\I,\I'$ homogeneous ideals generated in degree at least two.
Suppose for every vertex $u \in Q_0$, $x_{M_{uu}}^{uu} \in \cX^c$.
If $A \iso B$ as (ungraded) algebras, then $A \iso B$ as graded path algebras.
\end{lemma}
\begin{proof}
We maintain the notation introduced in Lemma \ref{lem.noloops}. In particular, we may assume  $M=M_Q=M_{Q'}$ and that $\phi$ is as in \eqref{eq.phi}. 

If for some vertex $u$ we have $\alpha_u \neq 0$, then we can replace $x_{M_{uu}}^{uu}$ and $y_{M_{uu}}^{uu}$ by $\alpha_u\inv x_{M_{uu}}^{uu}$ and $\alpha_u\inv y_{M_{uu}}^{uu}$, respectively, so that
$\phi(x_{M_{uu}}^{uu}) = f_u + y_{M_{uu}}^{uu} + (y_{M_{uu}}^{uu})'$ for some $(y_{M_{uu}}^{uu})' \in (B_+)^2$.

Let $p$ be a path with source $u$ and target $w$ consisting of arrows in $\cX$. Then $\phi(p) = q + q'$ for some path $q$ consisting of arrows in $\cY$ and some $q' \in (B_+)^2$. Set $a_u = x_{M_{uu}}^{uu}$ (resp. $b_u = y_{M_{uu}}^{uu}$) and similarly for $a_w$ and $b_w$.
By the above, $\phi(a_u)= f_u + b_u + b_u'$ and $\phi(a_w) = f_w + b_w + b_w'$ for some $b_u',b_w' \in (B_+)^2$. Thus,
\begin{align}
\label{eq.comm2}
\phi([p]) - [\phi(p)]	
	&= (f_u\phi(p)-\phi(p)f_w) + \text{(higher degree terms)}.
\end{align}
Hence, if $\deg(p)=d$, then $\phi([p]) - [\phi(p)] \in (B_+)^{d+1}$.

More generally, we write $r([\cX])$ (resp. $r([\cY])$) for a polynomial in the $[\cX]$ (resp. $[\cY]$).
If $r$ is a homogeneous relation of degree $d$ in the $[\cX]$, then by \eqref{eq.comm2},
\begin{align}
\label{eq.reln1}
r([\cY]) = r([\cY]) - \phi(r([\cX])) \in (B_+)^{d+1}.
\end{align}

Now let $r$ be any homogeneous relation in the $\{x_k^{uv}\}$ of path degree $d$ with source $u$ and target $v$. We claim that $r$ is also a relation in the $\{y_k^{uv}\}$ and hence $\dim (A_d)_{uv} = \dim (B_d)_{uv}$. This claim holds trivially for all $u,v$ when $d=0$ or $1$ so assume it holds for all $n < d$.

Let $a_u \in \cX^c$ be a loop at vertex $u$, which exists by hypothesis. By Lemma \ref{lem.comm1},
\[ r = \sum_{s=0}^d a_u^s r_s([\cX])\]
where each $r_s([\cX])$ is a homogeneous polynomial of degree $d-s$ in $[\cX]$ with source $u$ and target $v$. Let $t$ be the largest integer for which $r_t([\cX]) \neq 0$. 
If $t=0$, then $r$ is a relation in the $[\cX]$ and the claim holds by \eqref{eq.reln1}. Assume $t>0$ 
so that 
\[
0 	= \phi(r) 
	=  \sum_{s=0}^t (f_u + b_u + b_u')^s \phi( r_s([\cX]) )
\]
where $b_u=y_{M_{uu}}^{uu}$ and $b' \in (B_+)^2$.
Then $r_t([\cY]) \in (B_+)^{d-t+1}$ and so by homogeneity, $r_t([\cY])=0$.
This implies that $r_t([\cY])$ is a relation of degree $d-t$ from vertex $u$ to vertex $v$ and by 
our inductive hypothesis, $\dim(A_{d-t})_{uv} = \dim(B_{d-t})_{uv}$.
It follows that there exists a relation $r'$ in the $\{x_k^{uv}\}$ of degree $d-t$ from vertex $u$ to vertex $v$ such that $r'$ is nonzero in $\{y_k^{uv}\}$. This contradicts our hypothesis and so the claim holds for $n=d$.

Thus, $\dim (A_d)_{uv} \geq \dim (B_d)_{uv}$ for all $d$ and all choices of $u$ and $v$.
Applying $\phi\inv$ we obtain the reverse inequality and so $\dim (A_d)_{uv} = \dim (B_d)_{uv}$.
Hence, the map $\phi: A \rightarrow B$ given by $\phi(x_k^{uv}) = y_k^{uv}$ 
and $\phi(e_u)=f_u$ is a graded path algebra isomorphism.
\end{proof}

Our analysis so far has covered the two extreme cases: $\alpha_u=0$ for all $u \in Q_0$ or $\alpha_u \neq 0$ for all $u \in Q_0$. We now consider the intermediate case and prove our main theorem.

\begin{theorem}
\label{thm.main}
Let $A = \kk Q/\I$ and $B=\kk Q'/\I'$ with $\I,\I'$ homogeneous ideals generated in degree at least two.
If $A \iso B$ as (ungraded) algebras, then $A \iso B$ as graded path algebras.
\end{theorem}
\begin{proof}
We maintain the notation introduced in Lemmas \ref{lem.noloops} and \ref{lem.all-loops}. Let 
\[ U = \{ u \in Q_0 : \phi_0(x_{M_{uu}}^{uu}) = 0 \}\] 
and let $\oQ$ be the quiver obtained by attaching to $Q$ an additional loop $a_u$ at each $u \in U$. The ideal $\I$ of $\kk Q$ extends to an ideal of $\kk\oQ$. Define $\oA=\kk\oQ/\I$. We obtain $\overline{Q'}$ similarly by attaching a corresponding loop $b_u$ to $Q'$ for each $u \in U$ and define $\oB=\kk\overline{Q'}/\I'$. 

Note that we still have $\phi:A \to B$ as in \eqref{eq.phi}. We extend $\phi$ to a map $\ophi:\oA \to \oB$ by setting $\ophi(a_u)=f_u + b_u$ for each $u \in U$. Then $\ophi$ preserves the path algebra relations. For example, suppose $\overline{Q}$ has a loop $a_v$ at vertex $v \neq u$. Then according to Lemma \ref{lem.verts}, each summand of $\phi(a_v)$ has source and target $v$. Consequently, $\phi(a_ua_v) = (f_u + b_u)\phi(a_v) = 0$. Moreover, $\ophi$ (trivially) preserves the generators of $\I$ so that $\ophi$ is an isomorphism $\oA \to \oB$. 

In the notation of Lemma \ref{lem.noloops}, let $\overline{\cX}$ (resp. $\overline{\cY}$) be the set of generators in $\oA$ (resp. $\oB$) corresponding to $\cX$ (resp. $\cY$) in $A$ (resp. $B$). Define $\overline{\cX}^c$ and $\overline{\cY}^c$ similarly. Since $\ophi$ is an extension of $\phi$, then $\overline{\cX}=\cX$ and $\overline{\cY}=\cY$, while $\overline{\cX}^c = \cX^c \cup \{a_u\}_{u \in U}$ and $\overline{\cY}^c = \cY^c \cup \{b_u\}_{u \in U}$.

By Lemma \ref{lem.all-loops}, there is an isomorphism $\psi:\oA \to \oB$ which maps $x_k^{uv} \mapsto y_k^{uv}$. That is, $a_u \mapsto b_u$ for every $u \in U$. Restricting $\psi$ to the generators of $A$, i.e., $\cX \cup \cX^c$, it follows directly that $\psi$ is an isomorphism $\kk Q \to \kk Q'$. Let $r \in \I$ be a relation in the $\{x_k^{uv}\}$. The proof of Lemma \ref{lem.all-loops} shows that $r$ is a relation in the $\{y_k^{uv}\}$. By definition, $r$ does not involve those $a_u$ with $u \in U$. Hence, restricting $\psi$ as above gives a bijection $\I \to \I'$ and so $\psi$ restricts to a graded isomorphism $A \to B$.
\end{proof}

An initial application of Theorem \ref{thm.main} is the Zariski cancellation problem.
See \cite{BZ2,LWZ1} for more background on this problem for noncommutative algebras.

Let $Q$ be a quiver with $|Q_0|=n$ and $A=\kk Q/\I$ a graded path algebra.
Then the polynomial extension $A[t]$ is again a graded path algebra with adjacency matrix $M_Q+I_n$.
The following theorem is now almost immediate from \cite[Theorem 9]{BZ1}. 

\begin{theorem}
Let $A$ and $B$ be graded path algebras.
Suppose that $\cnt(A) \cap A_1 = \{0\}$.
If $A[t_1,\hdots,t_n] \iso B[s_1,\hdots,s_n]$ as ungraded algebras, then $A \iso B$.
\end{theorem}


Let $Q$ be the quiver below.
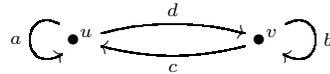
\[\xymatrix{ & \bullet^u  \ar@/^/[rr]^d \ar@(ul,dl)[]_a &  & \bullet^v \ar@/^/[ll]^c \ar@(ur,dr)[]^b & &}\]
If $\I = (a^2-dc,b^2-cd) \subset \kk Q$, then $\kk Q/\I \iso A\# G$ where
$A = \kk\langle u,v : u^2-v^2 \rangle$ and $G=\langle g \rangle$ is the cyclic group of order $2$
acting on $A$ by $g(u)=u$ and $g(v)=-v$.
For $q_1,q_2 \in \kk^\times$, set $C(q_1,q_2) = \kk Q/(ad-q_1db,ca-q_2bc)$.
Then $C(1,1) \iso B\#G$ where $B = \kk[u,v]$ and $G$ acts on $B$ by
$g(u)=u$ and $g(v)=-v$. It follows immediately from Theorem \ref{thm.main} that $A\# G \not\iso B\#G$.

\begin{proposition}
Let $p_1,p_2,q_1,q_2 \in \kk^\times$.
Then $C(q_1,q_2) \iso C(p_1,p_2)$ if and only if there exists $k \in \kk^\times$
such that $(p_1,p_2) = (kq_1,kq_2)$ or $(p_1,p_2) = (kq_1\inv,kq_2\inv)$.
\end{proposition}

\begin{proof}
By Theorem \ref{thm.main}, we may assume that there is an isomorphism of graded path algebras
$\phi: C(q_1,q_2) \rightarrow C(p_1,p_2)$. 
Thus, we have one of the following cases:
\begin{enumerate}
\item $\phi(a) = k_1 a$, $\phi(b) = k_2 b$, $\phi(c) = k_3 c$, $\phi(d) = k_4 d$,
\item $\phi(a) = k_1 b$, $\phi(b) = k_2 a$, $\phi(c) = k_3 d$, $\phi(d) = k_4 c$,
\end{enumerate}
for some $k_i \in \kk^\times$.

Suppose we are in the first case, then
\begin{align*}
	\phi( ad-q_1db ) &= k_4(k_1ad-q_1k_2db) = k_4(k_1p_1 - k_2 q_1)db, \\
	\phi( ca-q_2bc ) &= k_3(k_1ca-q_2k_2bc) = k_3(k_1p_2-k_2q_2)bc.
\end{align*}
Thus, $p_i = (k_2/k_1)q_i$. The second case is similar.
\end{proof}


\subsection*{Acknowledgments} 
The author thanks Dan Rogalski, Kent Vashaw, and Robert Won for helpful conversations,
and the referee for several suggestions that improved the exposition and corrected some errors from the initial draft. 

\bibliographystyle{amsplain}

\providecommand{\bysame}{\leavevmode\hbox to3em{\hrulefill}\thinspace}
\providecommand{\MR}{\relax\ifhmode\unskip\space\fi MR }
\providecommand{\MRhref}[2]{%
  \href{http://www.ams.org/mathscinet-getitem?mr=#1}{#2}
}
\providecommand{\href}[2]{#2}

\end{document}